\documentclass[12pt]{amsart}
\usepackage{amscd}
\usepackage{verbatim}
\usepackage{amssymb, amsmath, amsthm, amscd,ifthen}
\usepackage[dvips]{graphics}
\usepackage[cp866]{inputenc}
\usepackage{graphicx}
\usepackage{epsfig}
\usepackage{color}

\usepackage{amsmath,amssymb,amscd,}
\usepackage[all]{xy}

\usepackage{amssymb}

\newtheorem{thm}{Theorem}[section]

\newtheorem{prop}[thm]{Proposition}

\newtheorem{Ex}[thm]{Example}

\newtheorem{lemma}[thm]{Lemma}

\theoremstyle{definition}

\newtheorem{dfn}[thm]{Definition}

\title[Intersections on the moduli spaces of  flexible polygons]{Intersection numbers of Chern classes of tautological line bundles on the moduli spaces of flexible polygons}
\author{Ilia Nekrasov}

\address{St.~Petersburg State University, Chebyshev Laboratory}
  \email{geometr.nekrasov@yandex.ru}

\author{Gaiane Panina}

\address{St.~Petersburg Department of V.A. Steklov Institute of Mathematics of the Russian Academy of Sciences}
\email{gaiane-panina@rambler.ru}

\author{Alena Zhukova}

\address{St.~Petersburg State University, Faculty of Liberal Arts and Sciences}
\email{a.zhukova@spbu.ru}


\begin{document}
\begin{abstract}
Given a  flexible $n$-gon with generic side lengths, the moduli space of its configurations in $\mathbb{R}^2$ as well as in $\mathbb{R}^3$ is a smooth manifold. It is equipped with $n$ \textit{tautological} line bundles  whose definition is motivated by M. Kontsevich's  tautological bundles over $\mathcal{M}_{0,n}$. We study their Euler classes, first Chern classes and  intersection numbers, that is, top monomials in Chern (Euler) classes.   The latter are interpreted geometrically as the signed numbers of some \textit{triangular configurations} of the flexible polygon.
\end{abstract}

\maketitle \setcounter{section}{0}

\section{Introduction}\label{section_abstr}

Assume that an $n$-tuple of positive numbers $L=(l_1,...,l_n)$ is fixed. We associate with it a \textit{flexible polygon}, that is, $n$  rigid bars of lengths $l_i$ connected in a   cyclic chain by revolving joints. A \textit{configuration} of   $L$ is an $n$-tuple  of points $(q_1,...,q_n)$ with $|q_iq_{i+1}|=l_i, \ \ |q_nq_1|=l_n$.

It is traditional since long (see \cite{fa}, \cite{HausmannKnu}, \cite{Kam}  and many other papers and authors)
to study the following two spaces:

\begin{dfn}\label{defConfSpace} \textit{The moduli space
 $M_2(L)$}  is the set  of all planar configurations  of $L$
modulo  isometries of $\mathbb{R}^2$.

 \textit{The moduli space
 $M_3(L)$}  is the set  of all configurations  of $L$ lying in $\mathbb{R}^3$
modulo orientation preserving  isometries of $\mathbb{R}^3$.
\end{dfn}
\begin{dfn}\label{SecondDfn}
Equivalently, one defines
$$M_2(L)=\{(u_1,...,u_n) \in (S^1)^n : \sum_{i=1}^n l_iu_i=0\}/O(2), \hbox{ and} $$
$$M_3(L)=\{(u_1,...,u_n) \in (S^2)^n : \sum_{i=1}^n l_iu_i=0\}/SO(3). $$
\end{dfn}

The second definition shows that  $M_{2}(L)$  and $M_{3}(L)$ do not depend on the
ordering of $\{l_1,...,l_n\}$; however, they do depend on the values
of $l_i$.

Throughout the paper  we assume that no configuration of $L$ fits in a
straight line. This assumption implies that the moduli spaces $M_{2}(L)$  and $M_{3}(L)$
are smooth closed manifolds.
In more details, let us take all subsets $I\subset \{1,...,n\}$.
The associated hyperplanes
$$\sum_{i\in I}l_i =\sum_{i\notin I}l_i$$
called \textit{walls }subdivide the parameter space $\mathbb{R}_+^n$
into a number of \textit{chambers}. The topological type of $M_2(L)$  and $M_3(L)$  depends only on the chamber containing $L$;
this becomes  clear in view of the (coming below) stable configurations representations.
For $L$ lying strictly inside  a chamber, the spaces $M_{2}(L)$  and $M_{3}(L)$  are smooth manifolds.

Let us make an additional assumption: throughout the paper we assume that $\sum_I\pm l_i$ never vanishes for all non-empty $I\subset \{1,...,n\}$.
This agreement does not restrict generality: one may perturb the edge lengths while staying in the same chamber.
So for instance, when we write $L=(3,2,2,1,1)$, we mean $L=(3+\varepsilon_1,2+\varepsilon_2,2+\varepsilon_3,1+\varepsilon_4,1+\varepsilon_5)$  for some generic small epsilons.

The space $M_{2}(L)$  is an $n-3$-dimensional manifold. In most of the cases it is non-orientable,
so we  work with cohomology ring with coefficients in $\mathbb{Z}_2$.

The space $M_{3}(L)$  is a $2n-6$-dimensional complex-analytic manifold.\footnote{Moreover, Klyachko  \cite{Kl} showed that it is an algebraic variety.} So we work with cohomology ring with integer coefficients. Since $M_3(L)$ has a canonical orientation coming from the complex structure, we canonically  identify $H^{2n-6}(M_3(L),\mathbb{Z})$ with $\mathbb{Z}$.

\subsection*{Stable configuration of points}

We make use of yet another representation of  $M_{2}(L)$  and $M_{3}(L)$.
Following   paper \cite{KapovichMillson}, consider  configurations of $n$ (not necessarily all distinct)
points
$p_i$ in the real
 projective line $\mathbb{R}P^1$ (respectively, complex projective line).  Each point $p_i$  is assigned the weight
 $l_i$.  The
configuration of (weighted) points is called  {\em
stable} if sum of the weights of coinciding points is
less than half the weight of all points.

Denote by $S_\mathbb{R}(L)$ (respectively, $S_\mathbb{C}(L)$) the space of stable configurations in the real projective (respectively, complex projective) line.
The group $PGL(2,\mathbb{R})$  (respectively, $PSL(2,\mathbb{C})$)  acts naturally on this space.

In this setting we have:
$$M_2(L)=S_\mathbb{R}(L)/PGL(2,\mathbb{R}), \hbox{ and}$$
$$M_3(L)=S_\mathbb{C}(L)/PSL(2,\mathbb{C}).$$

Therefore we think of  $M_{2}(L)$  and $M_{3}(L)$ as compactifications of the spaces of $n$-tuples of distinct points on the projective line (either complex or real).  That is, for each $n$ we have a  finite series of compactifications of ${\mathcal{M}}_{0,n}$ (respectively, ${\mathcal{M}}_{0,n}(\mathbb{R})$) depending on the particular choice of the lengths $L$.

For  the Deligne-Mumford compactification $\overline{\mathcal{M}}_{0,n}$ and its real part $\overline{\mathcal{M}}_{0,n}(\mathbb{R})$
M. Kontsevich introduced the \textit{tautological line bundles} $L_i, \ i=1,...,n$. Their  first Chern classes (or Euler classes
for the space $\overline{\mathcal{M}}_{0,n}(\mathbb{R})$) are  called $\psi$-\textit{classes}  $\psi_1, \dots ,\psi_{n}$. It is known \cite{Kon} that the top degree monomials  in $\psi$-classes  equal  the  multinomial coefficients. That is,

$$\psi_1^{d_1} \smile \dots \smile \psi_{n}^{d_n}=\binom{n-3}{d_1\ d_2\ ...\ d_n}\hbox{ \ \ \  for } \sum_{i=1}^n d_i=n-3$$

In the present paper  we mimic the definition  of tautological line bundles   and define $n$ similar tautological line bundles over the spaces $M_2(L)$ and $M_3(L)$, compute their Euler and Chern classes  (Section \ref{SectTaut}), and   study their intersection numbers (Section \ref{SecMon}).
The latter depend on the chamber containing the length vector $L$  and amount to counting some special types of triangular configurations\footnote{Precise meaning is clarified later.}  of the flexible polygon.

Besides, we show that the Chern classes almost never vanish (Proposition \ref{PropNonVanish}).

Informally speaking, computation of Euler classes and their cup-products is a baby version of computation of Chern classes: they are the same  with the exception that for the Euler classes one does not care about orientation since the coefficients are $\mathbb{Z}_2$ .

Throughout the paper all the cup-product computations take place in the ring generated by (the Poincar\'{e} duals of)  \textit{nice submanifolds} (Section \ref{SectNice}).  As a technical tool, we present multiplication rules in the ring (Proposition  \ref{computation_rules}).

\medskip
Let us fix some notation.  We say that some of the edges of a configuration $\{l_iu_i\}_{i\in I}$ are \textit{parallel}  if $u_i=\pm u_j$ for $i,j \in I$.  If  $I=\{k,k+1,...,k+l\}$ is a consecutive family of indices, this means that  the edges lie on a line.

Two parallel edges are either codirected or oppositely directed.

\medskip

\medskip
\textbf{Acknowledgement.} This research is supported by the Russian Science
Foundation under grant 16-11-10039.

\section{Intersections of nice manifolds.}\label{SectNice}
The cohomology rings $H^*(M_2(L),\mathbb{Z}_2)$ and $H^*(M_3(L),\mathbb{Z})$  are described in \cite{HausmannKnu}.
However for the sake of the subsequent computations we  develop intersection theory in different terms. The  (Poincar\'{e} duals of) the introduced below nice submanifolds generate a ring which is sufficient for our goals.

\subsection*{Nice submanifolds of $M_2(L)$}

Let $i\neq j$ belong to $[n]=\{1,...,n\}$.

Denote by $(ij)_{2,L}$  the image of the natural embedding of the space \newline $M_2(l_i+l_j,l_1,...,\hat{l}_i,...,\hat{l}_j,...,l_n)$ into
the space $M_2(L)$.
 That is, we think of the configurations of the new $n-1$-gon as the configurations of $L$ with parallel codirected edges $i$ and $j$ \textit{frozen} together to a single edge of length $l_i+l_j$. Since the moduli space does not depend on the ordering of the edges, it is convenient to think that $i$ and $j$  are consecutive indices. The space $(ij)_{2,L}$  is a (possibly empty) smooth closed submanifold of $M_2(L)$.  We identify it with the  Poincar\'{e}
 dual cocycle   and write for short
 $$(ij)_{2,L} \in H^1(M_2(L),\mathbb{Z}_2).$$

Denote by $(i\,\overline{j})_{2,L}$   the image of the natural embedding of the space \newline $M_2(|l_i-l_j|,l_1,...,\hat{l}_i,...,\hat{l}_j,...,l_n)$ into
the space $M_2(L)$. Again we have  a smooth closed (possibly empty) submanifold.
Now we think of the configurations of the new polygon  as the configurations of $L$ with parallel oppositely directed edges $i$ and $j$  {frozen} together  to a single edge of length $|l_i-l_j|$.


We can freeze several collections of edges, and analogously define a nice submanifold labeled by the formal product   $$ (l\overline{m})_{2,L}\cdot(ij\overline{k})_{2,L}=(ij\overline{k})_{2,L}\cdot (l\overline{m})_{2,L}\in H^3(M_2(L),\mathbb{Z}_2).$$  All submanifolds arising this way  are called \textit{nice submanifolds of $M_2(L)$, } or just \textit{nice manifolds}  for short.

Putting the above more formally,
 each nice manifold is labeled by an unordered formal product $$(I_1\overline{J}_1)_{2,L}\cdot ...\cdot(I_k\overline{J}_k)_{2,L},$$ where $I_1,...,I_k, {J}_1,,,{J}_k$ are some disjoint subsets of $[n]$  such that each set $I_i\cup {J}_i$ has at least one  element.\footnote{For further computations is convenient to define also nice manifolds with $I_i\cup {J}_i$  consisting of one element. That is,  we set $(1)_{2,L}=M_2(L),\ \  (1)_{2,L}\cdot (23)_{2,L}=(23)_{2,L}$ etc. }
{By definition, the manifold $(I_1\overline{J}_1)_{2,L}\cdot ...\cdot(I_k\overline{J}_k)_{2,L}$ is the subset of $M_2(L)$ defined by the conditions:}

 (1)  $i,j \in I_k$ implies  $u_i=u_j$, and

 (2)  $i \in I_k, j\in J_k$ implies  $u_i=-u_j$.

{Note that $(I_1\overline{J}_1)_{2,L}\cdot ...\cdot(I_k\overline{J}_k)_{2,L}$  is the  intersection of $(I_i\overline{J}_i)$, $i=1,...,k$.}
\medskip

Some of nice manifolds might be empty and thus represent the zero cocycle. This depends on the values of $l_i$.

\subsection*{Nice submanifolds of $M_3(L)$}
By literally repeating the above we define  nice submanifolds of $M_3(L)$ as point sets. Since a complex-analytic manifold has a fixed  orientation  coming from the complex structure, each nice manifold  has a \textit{canonical} orientation.

By definition, the  \textit{relative orientation} of a nice manifold $(I\overline{J})_{3,L}$  coincides with its canonical orientation iff $$\sum _I l_i>\sum _J l_i.\ \ \ \ \ \ \ \ (*)$$

 Further,  the two orientation (canonical and relative ones) of a nice manifold$$(I_1\overline{J}_1)_{3,L}\cdot ...\cdot (I_k\overline{J}_k)_{3,L}$$  coincide iff  the above inequality $(*)$  fails  for even number of $(I_i\overline{J_i})$.

 From now on, by $(I\overline{J})_{3,L} \in H^*(M_3(L),\mathbb{Z})$  we mean  (the Poincar\'{e} dual of) the  nice manifold taken with its relative orientation, whereas $(I\overline{J})_{3,L}^{can}$   denotes the nice manifold with the canonical orientation.

\medskip

To compute cup-product of these cocycles (that is, the intersections of    nice manifolds) we  need the following rules. Since the rules are  the same for all
$L$ and for both dimensions $2$ and $3$ of the ambient Euclidean space, we omit subscripts in the following proposition.In the sequel, we use the subscripts only if the dimension of the ambient space matters.

\begin{prop}\label{rules}\textbf{(Computation  rules)}\label{computation_rules}
The following  rules are valid   for nice submanifolds of $M_2(L)$ and  $M_3(L)$.
\begin{enumerate}
\item The cup-product is a commutative operation.
\item $(I\overline{J})=-(J\overline{I}).$
  \item If  the factors have no common entries, the cup-product equals the formal product, e.g.:
  $$(12) \smile (34)=(12)\cdot (34).$$

  \item If $I_1\cap I_2=\{i\},\ \ I_1\cap J_2= \emptyset, \ \ I_2\cap J_1= \emptyset,\ \ I_2\cap J_2= \emptyset, \hbox{and} \ \ J_1\cap J_2= \emptyset$,  then
$$(I_1\overline{J}_1)\smile (I_2\overline{J}_2)=
   (I_1\cup I_2 \ \overline{J_1\cup J_2}).
 $$
Examples: $$(123)\smile (345)=(12345),$$$$(123)\smile (34\overline{5})=(1234\overline{5}).$$
 \item If $J_1\cap J_2=\{i\},\ \ I_1\cap J_2= \emptyset, \ \ I_2\cap J_1= \emptyset,\ \ I_2\cap J_2= \emptyset, \hbox{and} \ \ I_1\cap I_2= \emptyset,$  then
$$(I_1\overline{J}_1)\smile (I_2\overline{J}_2)=-
   (I_1\cup I_2 \ \overline{J_1\cup J_2}).
 $$

Example:
  $$(12\overline{3})\smile (45\overline{3})=  -(1245\overline{3}).$$

\end{enumerate}

\end{prop}
\begin{proof} The statement (1) is true for $M_2$ since we work over $\mathbb{Z}_2$.  (1) is true also for $M_3$  since  the dimension  of a nice manifold is even. The statement (2) follows from the definition. The statement (5) follows from (2) and (4).
The statement (3) follows from  $(I_1\overline{J}_1)^{can} \smile (I_2\overline{J}_2)^{can}= ((I_1\overline{J}_1)\cdot (I_2\overline{J}_2))^{can}$, which is true by reasons of toric geometry, see \cite{HausmannKnu}.


So it remains to prove  (4).
  In notation of Definition \ref{SecondDfn}, take $(u_1,...,u_{n-3})\in (S^2)^{n-3}$ as a coordinate system  on $M_3(L)$. It is well-defined on some connected dense subset  of $ M_3(L)$. Taken together, the standard orientations of each of the copies of $S^2$ give the canonical
orientation on $M_3(L)$. In other words, the  basis  of the tangent space  $(d u_1,du_2,du_3,...,du_{n-3})$ yields the canonical orientation.

Let us start with two examples.

(A) The nice manifold $(12)_{3,L}$ embeds
as $$(u_1,u_3,...,u_{n-3})\rightarrow (u_1,u_1, u_3,...,u_{n-3}).$$ The relative orientation is defined by the basis $(d u_1,d u_3,...,d u_{n-3})$  of the tangent space. It always coincides with the canonical orientation.

              (B) The nice manifold $(1\overline{2})_{3,L}$  embeds
as $$(u_1,u_3,...,u_{n-3})\rightarrow (u_1,-u_1, u_3,...,u_{n-3}).$$ The relative orientation is defined by the basis $(d u_1,d u_3,...,d u_{n-3})$  of the tangent space. It coincides with the canonical orientation iff $l_1>l_2$.

Note that the nice manifold $(2\overline{1})_{3,L}$ coincides with  $(1\overline{2})_{3,L}$ as a point set, but comes with the opposite relative orientation defined by the basis $(du_2,du_3,...,du_{n-3})$.

 A nice manifold $(I\overline{J})$  such that none of $n,n-1,n-2$ belongs to $I\cup J$  embeds in a similar way.
Assuming that $1\in I$, the relative orientation is defined by the basis $$(d u_1,\{du_i\}_{i\notin I\cup J\cup\{n,n-1,n-2\}}).$$

For a nice manifold $(I\overline{J})$  such that   $I\cup J$ has less than $n-3$ elements, one chooses another coordinate system obtained by renumbering of the edges.

The statement (4) is now straightforward   if there are at least three edges that do not participate in the labels of nice manifolds.

Now prove (4) for the general case. We may assume that $n\notin I_1\cup I_2\cup J_1\cup J_2$.
\textit{Defreeze} the edge  $n$: cut it in three smaller edges and join the pieces by additional revolving joints.
Thus we obtain a new linkage $$L'=(l_1,...,l_{n-1}, \frac{1}{3}l_n,\frac{1}{3}l_n,\frac{1}{3}l_n).$$
The nice manifolds $(I_1\overline{J}_1)_{3,L}\smile (I_2\overline{J}_2)_{3,L}$  and \newline $(I_1\overline{J}_1)_{3,L'}\smile (I_2\overline{J}_2)_{3,L'}\smile(n\ n+1\ n+2)_{3,L'}$  have one and the same relative orientation.  For $(I_1\overline{J}_1)_{3,L'}\smile (I_2\overline{J}_2)_{3,L'}$ we can apply (4) and write

$$(I_1\overline{J}_1)_{3,L}\smile (I_2\overline{J}_2)_{3,L}=(I_1\overline{J}_1)_{3,L'}\smile (I_2\overline{J}_2)_{3,L'}\smile(n\ n+1\ n+2)_{3,L'}=$$
$$ \Big((I_1\overline{J}_1)_{3,L'}\cdot (I_2\overline{J}_2)_{3,L'}\Big)\smile(n\ n+1\ n+2)_{3,L'}=$$
$$(I_1\cup I_2\ \overline{J_1\cup J_2})_{3,L'}\smile(n\ n+1\ n+2)_{3,L'}=(I_1\cup I_2\ \overline{J_1\cup J_2})_{3,L}.$$
\end{proof}

\section{Tautological line bundles over $M_{2}$  and $M_{3}$.  Euler  and Chern  classes.}\label{SectTaut}
Let us give the main definition in notation of  Definition \ref{SecondDfn}:
\begin{dfn}
\begin{enumerate}
  \item  The tautological line bundle $E_{2,i}(L) $ is the real line bundle over the space $M_2(L)$  whose fiber over a point $(u_1,...,u_n)\in (\mathbb{R}P^1)^n$ is the tangent line to  $\mathbb{R}P^1$
at the point $u_i$.

  \item
Analogously, the tautological line bundle
 $E_{3,i}(L)$ is the complex line bundle over the space $M_3(L)$  whose fiber over a point $(u_1,...,u_n)\in (\mathbb{C}P^1)^n$ is the complex tangent line to the complex projective line $\mathbb{C}P^1$
at the point $u_i$.
\end{enumerate}

\end{dfn}

\begin{lemma}The bundles $E_{2,i}(L)$ and  $E_{2,j}(L)$ are isomorphic for any $i,j$.
\end{lemma} \textit{Proof.} {Define $\widetilde{M}_2(L):=\{(u_1,...,u_n) \in (S^1)^n : \sum_{i=1}^n l_iu_i=0\} $. That is, we have $M_2(L)=\widetilde{M}_2(L)/O(2)$. The space $\widetilde{M}_2(L)$ comes equipped with the  line bundles $\widetilde{E}_{2,i}(L)$:  the  fiber of $\widetilde{E}_{2,i}(L)$ over a point $(u_1,...,u_n)$ is the tangent line to  $\mathbb{R}P^1$ at the point $u_i$. Clearly, we have $E_{2,i}(L)=\widetilde{E}_{2,i}(L)/O(2).$  Let $\rho$ be the (unique) rotation of the circle $S^1$ which takes $u_i$ to $u_j$. Its pushforward $d\rho: \widetilde{E}_{2,i}(L)\rightarrow \widetilde{E}_{2,j}(L)$ is an isomorphism. Since $d\rho$ commutes with the action of $O(2)$, it yields an isomorphism between $E_{2,i}(L)$ and  $E_{2,j}(L)$. \qed}

\medskip

The bundles $E_{3,i}(L)$ and   $E_{3,j}(L)$  are (in general) not isomorphic, see Lemma  \ref{Lemma_for_quadrilateral} and further examples.


\medskip

\begin{thm}\label{ThmEulerChern}For $n\geq 4$ we have:\begin{enumerate}
                     \item

                     \begin{enumerate}
                       \item The   Euler class of $E_{2,i}(L)$ does not depend on $i$ and  equals\footnote{We omit in the notation dependence on the $L$, although $Ch(i)$ as well as $e$ depend on the vector $L$.} $$e:=e(E_{2,i}(L))=(12)_{2,L}+(1\overline{2})_{2,L}=(ij)_{2,L}+(i\,\overline{j})_{2,L}  \in H^1(M_2(L),\mathbb{Z}_2) \hbox{ for any} \ i \neq j.$$
                       \item The alternative expression for the Euler class is: $$e=(jk)_{2,L}+(kr)_{2,L}+(jr)_{2,L}  \in H^1(M_2(L),\mathbb{Z}_2) \hbox{ for any distinct }  j,k,r.$$
                     \end{enumerate}

                     \item  The first Chern class of $E_{3,i}(L)$ equals $$Ch(i):=Ch(E_{3,i}(L))=
                   (ij)_{3,L}-(i\,\overline{j})_{3,L} \in H^2(M_3(L),\mathbb{Z}) \hbox{ for any }  j \neq i.$$

                   \end{enumerate}
\end{thm}
\begin{proof}
Let us remind the reader briefly \textbf{a recipe } for
computing the first Chern class {of a complex line bundle. The background for the recipe comes from \cite{MilSt} and its explicitation \cite{Kaz}: in our case, the first Chern class equals the Euler class of the complex line bundle, so it is represented by the zero locus of a generic smooth section.} Assume we have a complex line bundle $E$  over an oriented smooth base $B$.
Replace $E$ by an $S^1$-bundle by taking the oriented unit  circle in each of the fibers.
\begin{itemize}
  \item If the dimension of the base is $2$,  the choice of orientation identifies $H^2(B,\mathbb{Z})$ with $\mathbb{Z}$. So  the Chern class is identified with some integer number $Ch(E)$ {which is called the Chern number, or the Euler-Chern number}. Choose a section $s$ of the $S^1$-bundle which is discontinuous at a finite number of singular points. {The singularities of the section correspond to the zeros of the associated section of $E$.} Each point contributes to $Ch(E)$  a summand.   Take one of such points $p$ and a small  positively oriented  circle $\omega \subset B$ embracing $p$. We may assume that the circle is small enough to fit in a neighborhood of $p$ where the bundle is trivializable. This means that in the neighborhood there exist a continuous section $t$.  Except for $t$, at each point of $\omega$ we have the  section $s$.
The point $p$ contributes to $Ch(E)$ the winding number of $s$ with respect to $t$. {For removable singularities this contribution is zero.}
      \item  The general case reduces to the two-dimensional one. Let the  dimension of the base be greater than $2$.  Choose a section $s$ of the $S^1$-bundle which is discontinuous at a finite number of { (not necessarily disjoint) }oriented submanifolds $m_i$ \footnote{ {The zero locus of a  generic smooth section of $E$ gives an example of such $m_i$. }} of real codimension $2$.  {In our case, we'll present these submanifolds explicitly; they are disjoint.} Each $m_i$ contributes  to the cocycle $Ch(E)$ the summand $c_i\cdot m_i$ where the integer number $c_i$  is computed as follows. Take a $2$-dimensional smooth oriented {disc} $\mathcal{D}\subset B$  transversally intersecting the manifold $m_i$ at a smooth point $p$.  We assume that  taken together, orientations of $m_i$ and $\mathcal{D}$ yield the (already fixed) orientation of $B$.  The pullback of the line bundle to $\mathcal{D}$ comes with the restriction of the section $s$, which we denote by $\overline{s}$. Let $c_i$ be the number computed in the previous section related to $\mathcal{D}$ and to the section $\overline{s}$ at the point $p$. Then the first Chern class equals (the Poincar\'{e} dual of) \ $Ch(E) = \sum_i c_im_i$.
                                                                 \end{itemize}

\medskip

The proof starts similarly for all the cases (1,a), (1,b) and (2). Let us describe a  section of $E_{2,i}(L)$ and $E_{3,i}(L)$ respectively. Its zero locus (taken with a weight $\pm 1$, coming from mutual orientations) is the desired  class, either   $e(i)$  or $Ch(i)$.

Fix any $j\neq i$, and consider the following (discontinuous) section of $E_{2,i}(L)$ or $E_{3,i}(L)$ respectively: at a point $(u_1,...,u_n)$   we take the unit vector in the tangent line (either real or complex) at $u_i$ pointing in the direction of the shortest arc connecting $u_i$ and $u_j$.  The section is well-defined except for the points of $M(L)$ with $u_i=\pm u_j$, that is, except for nice manifolds $(ij)$ and $(i\,\overline{j})$. The section is transversal to the zero section.

(1,a) To compute $e(i)$, it suffices to observe that the section has no continuous extension  neither to  $(1i)_{2,L}$ nor to $(1\overline{i})_{2,L}$, so the statement is proven.

(1,b) Each (discontinuous) choice of orientation of $S^1$  induces a discontinuous  section of $E_{2,i}(L)$. An orientation of the circle $S^1$ is defined by any three distinct points, say, $(u_j,u_k,u_r)$. Consider the section whose orientation agrees with the orientation of the ordered triple $(u_j,u_k,u_r)$. {Whenever any two of  $u_j,u_k,u_r$ coincide, we obtain a non-removable singularity of the section, so the statement follows.}

(2)  To compute $Ch(i)$, we need to take orientation into account.  We already know that $Ch(E_{3,i}(L))=A\cdot (ij)_{3,L}-B\cdot (i\,\overline{j})_{3,L}$ for some {integer (possibly zero)} $A$ and $B$ that may depend on $L$, $i$ and $j$.

\medskip

We are going to apply the above recipe. Therefore  we first look at the case when the base has real dimension two, which corresponds to $n=4$.
All existing cases are described in the following lemma:
\begin{lemma}\label{Lemma_for_quadrilateral}
{Theorem \ref{ThmEulerChern} (2) is valid for $n=4$. More precisely,}
\begin{enumerate}
  \item For $l_1,l_2,l_3,l_4$ such that $l_1>l_2>l_3>l_4$ and $l_3+l_2>l_4+l_1$, we have
 $Ch(1)=Ch(2)=Ch(3)=0$,  and \ \ $Ch(4)=2$.

   \item For $l_1,l_2,l_3,l_4$ such that $l_1>l_2>l_3>l_4$ and $l_3+l_2<l_4+l_1$, we have
  $Ch(2)=Ch(3)=Ch(4)=1$, and \ \ $Ch(1)=-1$.

\end{enumerate}
\end{lemma}\textit{Proof of the lemma.}
{Let us compute $Ch(1)$. Consider the section
of $E_{3,1}(L)$ equal to the unit vector $T_{u_1}(S^2)$ pointing in the direction of the shortest arc connecting $u_1$ and $u_2$.
The section is well-defined everywhere except for the points $(12)_{3,L}$ and $(1\overline{2})_{3,L}$. The manifold $M_3(L)$ is diffeomorphic to $S^2$ and can be parameterized by coordinates $(d, \alpha)$, where $d$  is the length of the diagonal $q_1q_3$ in a configuration $(q_1,q_2,q_3,q_4)$, and $\alpha$ is the angle between the (affine hulls of) triangles $q_1q_2q_3$ and $q_3q_4q_1$.
  The contribution of $(12)_{3,L}$ to the Chern number is computed according to the above recipe. Take a small circle on $M_3(L)$ embracing $(12)_{3,L}$ as follows: make first two edges almost parallel and codirected, and  rotate the (almost degenerate) triangle $q_1q_2q_3$ with respect to $q_3q_4q_1$ keeping diagonal $q_1q_3$ fixed. It remains to count the winding number. The contribution of the $(1\overline{2})_{3,L}$ to the Chern number is found  analogously. }

{Although the lemma is proven, let us give one more insight. Assume we have $l_1>l_2>l_3>l_4$ and $l_3+l_2>l_4+l_1$. Let us we make use of the representation of configurations of stable points  and factor out the action of $PSL(2,\mathbb{C})$ by assuming that first three points $p_{1}, p_{2}$ and $p_{3}$ are $0,1$ and $\infty$ respectively.  This is always possible  since $p_1,p_2$, and $p_3$ never coincide pairwise. Thus, $M_3(L)$ is the two-sphere, and $E_{3,4}(L)$  is the tangent bundle over $S^2$, whose Chern number count is a classical exercise in topology courses. Let us just repeat that in this representation we take the (same as above) section of the $E_{3,4}(L)$ as the unite vector in the tangent plane going in the direction of the point $p_1 = 0$. \qed}

\medskip

Now we prove the statement (2) of the theorem for the general case.
Without loss of generality, we may assume that $i=1$.
 Choose a generic point $p\in (12)_{3,L}$  (or a generic point $p\in (1\overline{2})_{3,L}$). It is some  configuration  $(q_1,\dots,q_n)$. Freeze  the edges  $4,5,...,n$. Fix also the rotation parameter with respect to the diagonal $q_1q_4$  (for instance, one may fix the angle between the planes $(q_1q_2q_3)$ and $(q_1q_4q_5)$; generically these triples are not collinear).  Keeping in mind the recipe, set { $\widetilde{\mathcal{D}}\subset M_3$ be the set of configurations satisfying all these freezing conditions. By construction, $p $ lies in $\widetilde{\mathcal{D}}$. The manifold $\widetilde{\mathcal{D}}$ amounts to the configuration space of a $4$-gon constituted of first three edges and the diagonal $q_1q_4$. Let the two-dimensional disc $\mathcal{D}$ be a neighborhood of $p$ in $\widetilde{\mathcal{D}}$. }

It remains to  refer to the case $n=4$. Theorem is proven.

Let us comment on the following alternative proof of $(1,a)$. The space $M_2(L)$ embeds in $M_3(L)$. The representation  via stable points configurations shows that $M_2(L)$  is the space of real points of the complex manifold $M_3(L)$; see \cite{HausmannKnu, Kl} for more details.
The bundle $E_{2,i}(L)$ is the real part of the pullback of $E_{3,i}(L)$. Therefore, the Euler class $e$ is the real part of the pullback of $Ch(i)$. In our framework, to take the pullback of $Ch(i)$ means to intersect its dual class with $M_2(L)$. It remains to observe that $(IJ)_{3,L}\cap M_2(L)=(IJ)_{2,L}$ and $(I\overline{J})_{3,L}\cap M_2(L)=(I\overline{J})_{2,L}$. One should keep in mind that passing to $M_2(L)$  we replace the coefficient ring $\mathbb{Z}$ by $\mathbb{Z}_2$.
\end{proof}

\medskip

\textbf{Remark.}  We have a natural  inclusion $incl:M_3(L)\rightarrow (\mathbb{C}P^1)^n/PSL(2,\mathbb{C})$. Define the analogous linear bundles over $(\mathbb{C}P^1)^n/PSL(2,\mathbb{C})$ and their first Chern classes  $\mathbf{Ch}(i)$. The above computation of the Chern class is merely taking the pullback $Ch(i)=incl^*\mathbf{Ch}(i)$.

\medskip
Let us now examine  the cases when the Chern class $Ch(i)$ is zero.

\begin{prop}\label{PropNonVanish} Assume that $l_1>l_2>...>l_n$. \begin{enumerate}
                                              \item If $$l_2+l_3<l_1+l_4+l_5+...+l_n \ \ \ \ \ \  (**)$$ then the Chern class $Ch(i)$ does not vanish for all $i=1,...,n$.

                                              \item The above condition  $(**)$ holds true for all the chambers (that correspond to non-empty moduli space) except the unique one  represented by $L=(1,1,1,\varepsilon,...,\varepsilon)$.
                                                  In this exceptional case $Ch(1)=Ch(2)=Ch(3)=0$, and  $Ch(i)\neq 0$ for $i>3$.
                                            \end{enumerate}

\end{prop}
\begin{proof}(1) Take a maximal (by inclusion) set $I$  containing $1$ with the property $\sum_{i\in I} l_i<\sum_{i\notin I} l_i.$  \footnote{Such a set may be not unique; take any.}  The condition $  (**) $  implies that its complement has at least three elements. Make all the edges from $I$ codirected and freeze them together to a single edge. Also freeze (if necessary) some of the remaining edges to get a flexible $4$-gon out of the initial $n$-gon. It is convenient to think that the edges that get frozen are consecutive ones. We get a new flexible polygon $L'$ whose moduli space $M_3(L')$ embeds in $M_3(L)$. For a given $i$ choose $j\in [n]$ in such a way that the edges  $i$ and $j$ are not frozen together, and write $Ch(i)=(ij)-(i\,\overline{j})$.

(a) If none of $i,j$ is frozen with the edge  $1$, then $(i\,\overline{j})$  does not intersect $M_3(L')$, whereas $(i{j})$  intersects $M_3(L')$ transversally at exactly one point.

(b) If one of $i,j$ is frozen with the edge  $1$, then $(i{j})$  does not intersect $M_3(L')$, whereas $(i\,\overline{j})$ intersects $M_3(L')$  transversally at exactly one point.

In both cases the product of $Ch(i)$ with the cocycle $M_3(L')$ is non-zero, so the proof is completed.

(2) Clearly, $Ch(1)=(12)-(1\overline{2})=0$  since both summands are empty manifolds.
Further, $$Ch(4)\smile...\smile Ch(n)=\left[(41)+(1\overline{4})\right]\smile...\smile \left[(n1)+(1\overline{n})\right]=2^{n-3}\neq 0,$$ which proves the statement.
\end{proof}

\section{Monomials in Euler and Chern classes}\label{SecMon}

Let us start with small examples.
	Table \ref{Tab_Pentagon1} represents  the multiplication table for the five Chern classes for the flexible pentagon $L=(3,1, 1,1,1)$.
	
	\begin{table}[h] \caption{Multiplication table for  $L=(3,1, 1,1,1)$}
		\label{Tab_Pentagon1}
		\begin{tabular}{cccccc}
			& $Ch(1)$ & $Ch(2)$  & $Ch(3)$ & $Ch(4)$  & $Ch(5)$   \\
			$Ch(1)$ & 1 & -1 & -1 & -1& -1\\
			$Ch(2)$ & -1 & 1 & 1 & 1& 1\\
			$Ch(3)$ & -1 & 1 & 1 & 1& 1\\
			$Ch(4)$ & -1 & 1 & 1 & 1& 1\\
			$Ch(5)$ & -1 & 1 & 1 & 1 & 1\\
		\end{tabular}
	\end{table}
Here is a detailed computation of $Ch(1)\smile Ch(2)$:
$$Ch(1)\smile Ch(2)=[(13)+(\overline{1}3)]\smile [(23)+(\overline{2}3)]=$$
$$(123)+(1\overline{2}3)+(\overline{1}23)+(\overline{1}\overline{2}3)=0+0-1+0=-1.$$
	
Tables \ref{Tab_Pentagon2}, \ref{Tab_Pentagon3}, \ref{Tab_Pentagon4}, \ref{Tab_Pentagon5}, \ref{Tab_Pentagon6} represent  the multiplication tables for all the remaining cases of flexible pentagons (listed in  \cite{Zvonk}).
	
	\begin{table}[h] \caption{Multiplication table for  $L=(2,1, 1,1, \varepsilon)$}
		\label{Tab_Pentagon2}
		\begin{tabular}{cccccc}
			& $Ch(1)$ & $Ch(2)$  & $Ch(3)$ & $Ch(4)$  & $Ch(5)$   \\
			$Ch(1)$ & 0 & 0 & 0 & 0& -2\\
			$Ch(2)$ & 0 & 0 & 0 & 0& 2\\
			$Ch(3)$ & 0 & 0 & 0 & 0& 2\\
			$Ch(4)$ & 0 & 0 & 0 & 0& 2\\
			$Ch(5)$ & -2 & 2 & 2 & 2 & 0\\
		\end{tabular}
	\end{table}

\begin{table}[h] \caption{Multiplication table for  $L=(3,2, 2, 1, 1)$.}
	\label{Tab_Pentagon3}
	\begin{tabular}{cccccc}
		& $Ch(1)$ & $Ch(2)$  & $Ch(3)$ & $Ch(4)$  & $Ch(5)$   \\
		$Ch(1)$ & -1 & 1 & 1 & -1& -1\\
		$Ch(2)$ & 1 & -1 & -1 & 1& 1\\
		$Ch(3)$ & 1 & -1 & -1 & 1& 1\\
		$Ch(4)$ & -1 & 1 & 1 & -1& 3\\
		$Ch(5)$ & -1 & 1 & 1 & 3 & -1\\
	\end{tabular}
\end{table}

\begin{table}[h]\caption{Multiplication table for $L=(2, 2,1,1,1)$.}
	\label{Tab_Pentagon4}
	\begin{tabular}{cccccc}
		 & $Ch(1)$ & $Ch(2)$  & $Ch(3)$ & $Ch(4)$  & $Ch(5)$   \\
		$Ch(1)$ & -2 & 2 & 0 & 0& 0\\
		$Ch(2)$ & 2 & -2 & 0 & 0& 0\\
		$Ch(3)$ & 0 & 0 & -2 & 2& 2\\
		$Ch(4)$ & 0 & 0 & 2& -2& 2\\
		$Ch(5)$ & 0 & 0 & 2 & 2& -2\\
	\end{tabular}
\end{table}

	\begin{table}[h]\caption{Multiplication table for $L=(1,1,1,1,1)$.}
		\label{Tab_Pentagon5}
		\begin{tabular}{cccccc}
			& $Ch(1)$ & $Ch(2)$  & $Ch(3)$ & $Ch(4)$  & $Ch(5)$   \\
			$Ch(1)$ & -3 & 1 & 1 & 1& 1\\
			$Ch(2)$ & 1 & -3 & 1 & 1& 1\\
			$Ch(3)$ & 1 & 1 & -3 & 1& 1\\
			$Ch(4)$ & 1 & 1 & 1 & -3& 1\\
			$Ch(5)$ & 1 & 1 & 1 & 1& -3\\
		\end{tabular}
	\end{table}
	
\begin{table}[h] \caption{Multiplication table   for  $L=(1,1, 1,\varepsilon, \varepsilon)$}
		\label{Tab_Pentagon6}
		\begin{tabular}{cccccc}
			& $Ch(1)$ & $Ch(2)$  & $Ch(3)$ & $Ch(4)$  & $Ch(5)$   \\
			$Ch(1)$ & 0 & 0 & 0 & 0& 0\\
			$Ch(2)$ & 0 & 0 & 0 & 0& 0\\
			$Ch(3)$ & 0 & 0 & 0 & 0& 0\\
			$Ch(4)$ & 0 & 0 & 0 & 0& 4\\
			$Ch(5)$ & 0 & 0 & 0 & 4 & 0\\
		\end{tabular}
	\end{table}

\medskip
\newpage
\textbf{Remarks.} In each of the tables, the parity of all the entries is one and the same. Indeed, modulo $2$ these computations are the computations of the squared Euler class.

Table \ref{Tab_Pentagon6}  illustrates  the  case $2$ of Proposition \ref{PropNonVanish}.

\subsection{First computation of $e^{n-3}$}

If the dimension of the ambient space is $2$, all the tautological linear bundles are isomorphic, so we have the unique  top degree monomial
 $e^{n-3} \in \mathbb{Z}/2$. Let us compute it using rules from  Proposition \ref{rules}:$$e^2=[(12)+(1\overline{2})]\smile [(23)+(2\overline{3})]=(123)+(1\overline{2}\overline{3})+(12\overline{3})+(1\overline{2}{3}),$$

$$e^3=e^2\smile [(34)+(3\overline{4})]=(1234)+(1\overline{2}34)+(12\overline{3}4)+$$$$(123\overline{4})+(1\overline{23}4)+
(12\overline{34})+(1\overline{2}3\overline{4})+(1\overline{234}).$$

Proceeding this way one concludes:

\begin{prop} \label{PropEulMon}  \begin{enumerate}
                 \item The top power of the Euler class $e^{n-3}$  (as an element of $\mathbb{Z}_2$)
equals the number of triangular configurations of  the flexible polygon $L$ such that all the edges $1,...,n-2$ are parallel.
                 \item Choose any three vertices of  the flexible polygon $L$. Let them be, say, $q_i,q_j,$ and $q_k$ for some $i<j<k$.
                 The top power of the Euler class $e^{n-3}$  
equals the number of triangular configurations of  the flexible polygon $L$  with the vertices $i,j$, and $k$. More precisely, we count  configurations  such that
\begin{enumerate}
  \item the edges $i+1,...,j$ are parallel,
  \item  the edges $j+1,...,k$ are parallel,
  \item  the edges $k+1,...,n$  and $1,...,i$ are parallel.\qed
\end{enumerate}

               \end{enumerate}

\end{prop}

\begin{Ex}   Let $L=(1,1,\dots,1)$, that is, we have a flexible equilateral $(2s+3)$-gon.
The number of triangles indicated in Proposition \ref{PropEulMon}, (1) is $\binom{2s+1}{s}$. By Luke theorem,  modulo $2$ it equals $$ \prod_{t \geq 0} (s_{t} - s_{t-1} +1),$$ where $\{s_{t}\}_{t\geq 0}$ are digits of the binary  numeral system representation of $s$.

Finally, we get $$e^{2s}=\left\{
                           \begin{array}{ll}
                             1, & \hbox{if } s=2^r-1;\\
                             0, & \hbox{otherwise.}
                           \end{array}
                         \right.
$$
\end{Ex}

\subsection{Second computation of $e^{n-3}$}
Now we make use of Theorem \ref{ThmEulerChern}(1,b) :
$$ e = (i\; i+1)+(i+1\; i+2)+(i\; i+2).$$

\begin{prop}\begin{enumerate}
              \item We have
$$e^{k} = \sum_{T_1\cup T_2= [k+2]}  (T_{1})\cdot (T_{2}),$$
where the  sum runs over all unordered  partitions  of the set $[k+2]=\{1,...,k+2\}$ into  two nonempty sets $T_{1}$ and $T_{2}$.
              \item In particular,
$$e^{n-3} = \sum_{T_1\cup T_2=  [n-1]} (T_{1})\cdot (T_{2}),$$ where the  sum runs over all unordered partitions of  $[n-1]$ into  two nonempty disjoint sets $T_{1}$ and $T_{2}$.
            \end{enumerate}

\end{prop}
\begin{proof}

Let us first prove the following lemma:

\begin{lemma} \label{transposition_square}
We have
$$(ij)\smile (ij) = (ijk) + (ijl) + (ij)\cdot (kl) =  (ij)\smile e.$$
\end{lemma}
\textit{Proof of the lemma.} Perturb the manifold $(ij)$ by  keeping $u_i$ as it is and pushing $u_j$ in the direction defined by the orientation of the triple $(u_i, u_k, u_l)$. We arrive at a manifold $\widehat{(ij)}_{k,l}$ which represents the same cohomology class. The  manifolds $\widehat{(ij)}_{k,l}$  and  $(ij)$ intersect transversely, and their product is the subset of $(ij)$ where the above orientation is not defined. The last equation follows from the representation $e = (jk)+(jl)+(kl)$. The lemma is proven. \qed

Now we prove the proposition using induction by $k$. The base comes from Theorem \ref{ThmEulerChern}(1,b):
$$e = (12)+(13)+(23).$$
The last expression equals exactly all two-component partitions of the set $\{1,2,3\} = [3]$.

The induction step:
$$e^{k+1} = e^{k}\smile e =$$
$$= \left(\sum_{T_1\cup T_2=\{1, 2, \dots , k+2\}}(T_{1})\cdot (T_{2})\right) \smile \Big((k+1 \; k+2)+(k+1 \; k+3)+(k+2 \; k+3)\Big).$$
The statement now follows from Lemma \ref{transposition_square}.

Let us illustrate the second induction step:
\begin{align*}
e^2 = &[(12)+(23)+(13)] \smile [(23)+(24)+(34)] =\\
&(12)\smile (23)+(12)\smile (24)+(12)\smile (34)+(23)\smile (23)+(23)\smile (24)+\\
&+(23)\smile (34)+(13)\smile (23)+(13)\smile (24)+(13)\smile (34) =\\
&(123)+(124)+(12)\smile (34)+(23)\smile (23)+(234)+\\
&+(234)+(123)+(13)\smile (24)+(134) =\\
&(124)+(12)\cdot (34)+(234)+ (123)+(14)\cdot (23)+\\
&+(13)\cdot (24)+(134).\\
\end{align*}
 Each summand in the last expression corresponds to a two-element partition of the set $\{1,2,3,4\}$.
\end{proof}

\medskip

\subsection{Monomials in  Chern classes} From now on, we consider $\mathbb{R}^{3}$ as the ambient space.  Since we have  different representations for the Chern class, there are different ways of computing the intersection numbers. They lead to combinatorially different answers for one and the same monomial and one and the same $L$, but of course, these different counts give one and the same number.

The canonical orientation of $M_3(L)$ identifies $H^{2n-6}(M_3(L),\mathbb{Z})$  with $\mathbb{Z}$. Therefore the top monomials in Chern classes  can be viewed as integer numbers.

\begin{thm}
The top power of the Chern class $Ch^{n-3}(1)$ equals the signed number of triangular configurations of the flexible polygon $L$ such that all the edges $1,...,n-2$ are parallel.\footnote{Remind that the edges  $\{l_iu_i\}_{i\in I}$ are \textit{parallel}  if $u_i=\pm u_j$ for $i,j \in I$.}  Each such triangle  represents a one-point nice manifold $(I\overline{J})$  for some partition $[n-2]=I\cup J$ { with $1 \in I$}.  Each triangle is counted with the sign

$$(-1)^N \cdot \epsilon ,$$

where  $N= |J|$  is the cardinality of $J$,  and

$$\epsilon= \left\{
             \begin{array}{ll}
               1, & \hbox{if \ \  }  \sum_I l_i> \sum_Jl_j;\\
               -1, & \hbox{otherwise.}
             \end{array}
           \right.$$
The expressions for the other top powers $Ch^{n-3}(i)$ come from renumbering.
\end{thm}
\begin{proof} We have $$Ch(1)=(12)-(1\overline{2})=(13)-(1\overline{3})=(14)-(1\overline{4}), \hbox{ etc.}$$ Therefore

 $$Ch^2(1)= [(12)-(1\overline{2})]\smile [(13)-(1\overline{3})]= (123)-(12\overline{3})-(1\overline{2}3)+(1\overline{23}).$$

 $$Ch^3(1)= [(14)-(1\overline{4})]\smile Ch^2 ( 1)=$$$$(1234)-(1\overline{2}34)-(12\overline{3}4)-(123\overline{4})+(12\overline{34})+(13\overline{24})+(14\overline{23})-(1\overline{234}).$$
At the final stage we arrive at a number of zero-dimensional nice manifolds. Each comes with the sign $(-1)^N$.
Further, each of them corresponds either to  $1$  if the relative orientation agrees with the canonical orientation, or to $-1$.
This gives us $\epsilon$.
\end{proof}

\medskip

Figure \ref{FigInters1}  represents all the  triangular configurations that arise in computation of $Ch^2(1)$  for the  equilateral pentagon. The values of $N$ and $\epsilon$  are:  \newline (a) $N=2$, $\epsilon =-1$,  (b) $N=1$, $\epsilon =1$,  (c) $N=1$, $\epsilon =1$.

\begin{figure}[h]
\centering \includegraphics[width=10 cm]{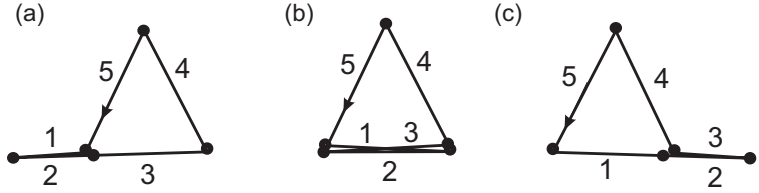}
\caption{Triangular configurations of the equilateral pentagon.}\label{FigInters1}
\end{figure}

\medskip

\begin{Ex} For the equilateral 2k+3-gon  we have  $$Ch^{2k}(i) =  (-1)^{k} \binom{2k}{k}-(-1)^{k+1} \binom{2k}{k-1}=(-1)^k\cdot\binom{2k+1}{k}.$$ Indeed, we count equilateral triangle configurations with either $k$ or $k-1$  edges codirected with the first edge.
\end{Ex}

For a three-term monomial the above technique gives:

             \begin{thm} Assume that  $d_1+d_2+d_3=n-3.$ The monomial

 $$Ch^{d_1}(1)\smile Ch^{d_2}(d_1+2)\smile Ch^{d_3 }(d_1+d_2+3)$$ equals the signed number of triangular configurations of  the flexible polygon $L$ such that \begin{enumerate}
  \item the edges $1,...,d_1+1$ are parallel,
  \item  the edges $d_1+2,...,d_1+d_2+2$ are parallel,
  \item  the edges $d_1+d_2+3,...,n$   are parallel.
\end{enumerate}

Each triangle is counted with the sign

$$(-1)^{N_1+N_2+N_3} \cdot \epsilon_1 \cdot \epsilon_2\cdot \epsilon_3,$$

{where  $N_i$   and $\epsilon_i$ refer to the $i$-th side of the triangular configuration.}

More precisely, each triangle represents  a one-point nice manifold  \newline $(I_1\overline{J}_1)\cdot(I_2\overline{J}_2)\cdot(I_3\overline{J}_3)$
with $1\in I_1,\ d_1+2 \in I_2,\ \ d_1+d_2+2 \in I_3$.
\newline Here $N_i=|J_i|,$ and $$\epsilon_i= \left\{
             \begin{array}{ll}
               1, & \hbox{if \ \  }  \sum_{I_i} l_k> \sum_{J_i}l_k;\\
               -1, & \hbox{otherwise.}
             \end{array}
           \right.$$
           The expressions for the other three-term top monomials come from renumbering.
\qed

\end{thm}

\medskip

Figure \ref{FigInters2}  represents two  triangular configurations that arise in computation of $Ch^2(1)\smile Ch^2(4)\smile Ch^2(7)$  for the  equilateral 9-gon. Here we have  \newline (a) $N_1=0 , N_2=1, N_3=0$, $\epsilon_1 =1,\epsilon_2 =1, \epsilon_3 =1$,  \newline (b) $N_1=1 , N_2=1, N_3=1$, $\epsilon_1 =1,\epsilon_2 =1, \epsilon_3 =1$.

\begin{figure}[h]
\centering \includegraphics[width=8 cm]{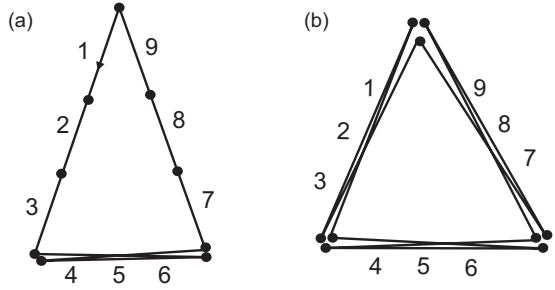}
\caption{Some triangular configurations of the equilateral $9$-gon.}\label{FigInters2}
\end{figure}
\medskip

The general case is:

\begin{thm}\label{main_theorem} Modulo renumbering any top monomial in Chern classes has the form $$Ch^{d_1}(1)\smile ...\smile Ch^{d_k}(k)$$ with $\sum_{i=1}^k d_i=n-3$  and $d_i \neq 0$ for $i=1,...,k$. Its value equals the signed number of triangular configurations of  the flexible polygon $L$ such that all the edges $1,...,n-2$ are parallel.  Each such triangle  represents a one-point nice manifold $(I\overline{J})$  for some partition $$[n-2]=I\cup J$$ with $k+1 \in I$. Each triangle is counted with the sign
	
	$$(-1)^N \cdot \epsilon ,$$
	
	where  $$N= |J|+\sum_{i \in J \cap [k]} d_i,$$  and
	
	$$\epsilon= \left\{
	\begin{array}{ll}
	1, & \hbox{if \ \  }  \sum_I l_i> \sum_Jl_j;\\
	-1, & \hbox{otherwise.}
	\end{array}
	\right.$$
	
\end{thm}
\begin{proof}

We choose the special way of representing the Chern classes  which is encoded in the graph depicted in Fig. \ref{FigGraph}. The vertices of the graph are labeled by elements of the set $[n]$. Each vertex $1,\dots,k$ has $d_k$ emanating directed edges.
A bold  edge  $\overrightarrow{ij}$ means that we choose the representation $$Ch(i)=(ij)-(i\overline{j}),$$
a dashed edge $\overrightarrow{ij}$ means that we choose the representation $$Ch(i)=(ij)+(\overline{i}j).$$

\begin{figure}[h]
\centering \includegraphics[width=10 cm]{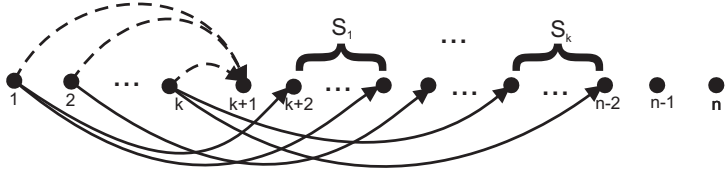}
\caption{Encoded  representations of $Ch(i)$}\label{FigGraph}
\end{figure}

Denote by $$S_i=\left\{ k+\sum_{j =1}^{i-1} d_j+2, \dots, k+\sum_{j =1}^{i} d_j+1  \right\}, $$
the set of vertices connected with the vertex $i$ by bold edges.

We multiply the Chern classes in three steps. Firstly, we multiply those that correspond to dashed edges and get

\begin{equation}\label{a}
Ch(1) \smile  Ch(2)\smile ... \smile  Ch(k)= \sum_{I \cup J = [k+1],\  k+1 \in I} (I \overline{J}).  \tag{A}
\end{equation}

Secondly, for every $i\in [k]$ we multiply the $d_{i}-1$ representations (of one and the same $Ch(i)$) that correspond to bold edges. This gives
$k$ relations
\begin{equation}\label{b}
Ch^{d_i-1}(i)=\sum_{I \cup J = S_i,\  i \in I} (-1)^{|J|} (I \overline{J}). \tag{B}
\end{equation}

Before we proceed note that:
\begin{enumerate}
\item Pick two nice manifolds, one from the sum (\ref{a}), and the other $(I \overline{J})$ from the sum (\ref{b}).  Assuming that  $I \cup J = S_i$, the unique common entry of the labels is $i$.
\item The labels of any two nice manifolds from the sum (\ref{b}) which are associated with different $i$'s are disjoint.
\end{enumerate}

Finally, one computes the product of (\ref{a}) and (\ref{b}) using the rules from Proposition \ref{computation_rules}. Every summand $(I \overline{J})$ in the result is a product of one nice manifold $(I_0 \overline{J_0})$ corresponding to a dashed edge, and $k$ nice manifolds $(I_1 \overline{J_1}), \dots,(I_k \overline{J_k})$ corresponding to the bold edges.

\medskip

Let us exemplify the last computation.
\begin{itemize}
\item For $n=6$:

$$Ch^2(1)\smile Ch(2)=[(13)+(\overline{1}3)]\smile [(23)+(\overline{2}3)] \smile[(14)-(1\overline{4}) ] \smile (2)=$$$$
[(123)+(\overline{1}23)+(1\overline{2}3)+(\overline{1}23)] \smile[(14)-(1\overline{4})]  =$$$$(1234)-(123\overline{4})+(1\overline{2}34)-(1\overline{2}3\overline{4})-(\overline{1}23\overline{4})
+(\overline{1}234)-(\overline{1}\overline{2}3\overline{4})+(\overline{1}\overline{2}34).$$

\item For $n=7$:

$$Ch^2(1)\smile Ch^2(2)=$$$$[(13)+(\overline{1}3)]\smile[(23)+(\overline{2}3)] \smile [(14)-(1\overline{4}) ] \smile [(25)-(2\overline{5})]=$$$$[(1234)-(123\overline{4})+(1\overline{2}34)-(1\overline{2}3\overline{4})-(\overline{1}23\overline{4})
+(\overline{1}234)-(\overline{1}\overline{2}3\overline{4})+(\overline{1}\overline{2}34)] \smile [(25)-(2\overline{5})]=$$$$(12345)-(1234\overline{5})-(123\overline{4}5)+(123\overline{4}\overline{5})-
(1\overline{2}34\overline{5})+(1\overline{2}345)+$$$$(1\overline{2}3\overline{4}\overline{5})-
(1\overline{2}3\overline{4}5)-(\overline{1}23\overline{4}5)+(\overline{1}23\overline{4}\overline{5})+(\overline{1}2345)-
(\overline{1}234\overline{5})+$$$$(\overline{1}\overline{2}3\overline{4}\overline{5})-
(\overline{1}\overline{2}3\overline{4}5)-(\overline{1}\overline{2}34\overline{5})+(\overline{1}\overline{2}345).$$
\end{itemize}

\end{proof}

\end{document}